\documentclass[12pt]{article}
\usepackage[utf8]{inputenc}

\usepackage{color}

\usepackage{epsfig}
\usepackage{cite}
\usepackage{latexsym,amsmath}
\usepackage{amsfonts}
\usepackage{float}
\usepackage{pict2e}
\usepackage{tikz}
\usepackage{caption}
\usepackage{amsfonts}
\usepackage{amssymb}
\usepackage{mathrsfs}
\usepackage{amsmath}
\usepackage{enumerate}
\usepackage{graphicx}
\usepackage{subfig}
\usepackage{MnSymbol}
\usepackage{mathtools}
\begin{document}
	
	\bibliographystyle{plain}
	
	\pagestyle{myheadings}
	\thispagestyle{empty}
	\newtheorem{theorem}{Theorem}[section]
	\newtheorem{corollary}[theorem]{Corollary}
	\newtheorem{definition}{Definition}
	\newtheorem{guess}{Conjecture}
	\newtheorem{claim}[theorem]{Claim}
	\newtheorem{problem}{Problem}
	\newtheorem{question}{Question}
	\newtheorem{lemma}[theorem]{Lemma}
	\newtheorem{proposition}[theorem]{Proposition}
	\newtheorem{observation}[theorem]{Observation}
	\newenvironment{proof}{\noindent {\bf
			Proof.}}{\hfill\rule{3mm}{3mm}\par\medskip}
	\newcommand{\remark}{\medskip\par\noindent {\bf Remark.~~}}
	\newcommand{\ch}{{\rm ch}}
	\newcommand{\de}{\em}
	\newtheorem{example}{Example}
	\newcommand{\set}[1]{\{#1\}}
	\newcommand{\norm}[1]{{|#1|}}
	\newcommand{\NN}{\mathbb{N}}
	
	\newcommand{\GS}{G/\mathcal{S}}
	\newcommand{\LS}{L_{\mathcal{S}}}
	\newcommand{\BS}{B_{\mathcal{S}}}
	\newcommand{\XS}{X_{\mathcal{S}}}
	\newcommand{\YS}{Y_{\mathcal{S}}}
	\newcommand{\mS}{\mathcal{S}}
	\newcommand{\MS}{M_{\mathcal{S}}}
	
	\title{\bf Bad list assignments for non-$k$-choosable $k$-chromatic graphs with $2k+2$-vertices}
	
	\author{Jialu Zhu\thanks{Department of Mathematics, Zhejiang Normal University, Email:jialuzhu@zjnu.edu.cn, }         \and
		Xuding Zhu\thanks{Department of Mathematics, Zhejiang Normal University, Email: xdzhu@zjnu.edu.cn,  Grant numbers: NSFC 11971438,U20A2068, ZJNSFC LD19A010001. }}
	
	\maketitle
	
	
	\begin{abstract}
		It was conjectured by Ohba, and proved by Noel, Reed and Wu  that $k$-chromatic graphs $G$ with $|V(G)| \le 2k+1$ are chromatic-choosable. This upper bound on $|V(G)|$ 
		is tight:  if $k$ is even, then  $K_{3 \star (k/2+1), 1 \star (k/2-1)}$ and $K_{4, 2 \star (k-1)}$ are $k$-chromatic graphs with $2 k+2$ vertices that are not chromatic-choosable. It was proved in [Zhu and Zhu, Minimum Non-chromatic choosable graphs with given chromatic number, arXiv:2201.02060] that these are the only non-$k$-choosable complete $k$-partite graphs with $2k+2$ vertices. 
		For   $G\in \{K_{3 \star (k/2+1), 1 \star (k/2-1)},K_{4, 2 \star (k-1)}\}$,
		a bad list assignment of $G$ is a $k$-list assignment $L$ of $G$ such that $G$ is not $L$-colourable. 
		Bad list assignments for $G=K_{4, 2 \star (k-1)}$ were characterized in [Enomoto, Ohba,  Ota and Sakamoto, Choice number of some complete multi-partite graphs, Discrete Mathematics 244 (2002) 55–66].  In this paper, we first give a simpler proof of this result, and then we characterize bad list assignments for $G=K_{3 \star (k/2+1), 1 \star (k/2-1)}$.  Using these results, we characterize all non-$k$-choosable   (non-complete) $k$-partite graphs  with $2k+2$ vertices.  
	\end{abstract}
	
	\noindent  {\bf Keywords:}  chromatic-choosable graphs, Ohba conjecture, bad list assignment.

	\section{Introduction}
	A \emph{proper  colouring} of a graph $G$ is a mapping $f:V(G) \rightarrow \mathbb{N}$ such that $f(u) \neq f(v)$ for every edge $uv$ of $E(G)$.  The {\em chromatic number } $\chi(G)$ of $G$ is the minimum $k$ such that $G$ has a proper $k$-colouring, i.e., a proper colouring $f$ with $|f(V(G))| \le k$. 
	\noindent A {\em $k$-list assignment } of a graph $G$ is a mapping $L$ which
	assigns to each vertex $v$ a set $L(v)$ of at least $k$ permissible colours, i.e., $|L(v)| \ge k$.  An {\em
		$L$-colouring} of $G$ is a proper   colouring   of $G$ which colours each vertex $v$ with
	a colour from $L(v)$. We say that $G$ is {\em
		$L$-colourable} if there exists an $L$-colouring of $G$, and $G$ is {\em $k$-choosable} if $G$ is $L$-colourable for each $k$-list assignment $L$ of $G$. More
	generally, for a function  $f: V(G) \to \NN$, an  {\em
		$f$-list assignment} of $G$ is a list assignment  $L$  with $|L(v)| \ge f(v)$ for all $v \in V(G)$. We say $G$ is {\em $f$-choosable} if $G$ is $L$-colourable for every $f$-list assignment $L$ of $G$.   The {\em choice number} $\ch(G)$ of $G$ is
	the minimum $k$ for which $G$ is $k$-choosable.   List colouring of
	graphs was introduced independently by Erd\H{o}s-Rubin-Taylor  \cite{ERT1980} and Vizing \cite{Vizing76}, and has been studied extensively in the literature  (cf. \cite{tuzasurvey}).
	
	It follows from the definitions that $\chi(G) \le ch(G)$ for any graph $G$, and it is well-known   \cite{ERT1980} that  bipartite graphs can have arbitrarily large choice number. 
	A graph $G$ is called {\em chromatic-choosable} if $\chi(G)=ch(G)$.  Some families of graphs are conjectured or proven to be   chromatic-choosable.  For example, the list colouring conjecture, proposed independently by Albertson and Collins, Gupta, and Vizing (see \cite{Hagg}),  asserts that line graphs are chromatic-choosable. This conjecture was proved   for bipartite graphs \cite{Gal}, for complete graphs of odd order \cite{Hagg} and for complete graphs of order $p+1$ for primes $p$ \cite{Schauz}.  Also it was conjectured by 
	Ohba \cite{Ohba2002} and proved by Noel, Reed and Wu \cite{NRW2015} that $k$-chromatic graphs $G$ with $|V(G)| \le 2 \chi(G)+1$ are chromatic-choosable.  We shall refer this result as   Noel-Reed-Wu Theorem in the remainder of this paper.
	
	We denote by $K_{k_1\star n_1,k_2 \star n_2, \ldots, k_q \star n_q}$ the complete  multi-partite graph with $n_i$ partite sets of size  $k_i$, for $i=1,2,\ldots, q$.  If $n_j=1$, then the number $n_j$ is omitted from the notation.
	For example, $K_{4, 2\star (k-1)}$ is the complete $k$-partite graph  with one partite set of size $4$, and $k-1$ partite sets of size $2$.  It was proved in \cite{EOOS2002} that if $k$ is an even integer, then $K_{4, 2\star (k-1)}$ and 
	$K_{3 \star (k/2+1), 1 \star (k/2-1)}$ are not $k$-choosable.
	
	So  Noel-Reed-Wu Theorem is tight, when $\chi(G)$ is even. 
	Noel \cite{Noel2013} conjectured that if $k$ is odd, then the bound is not tight, i.e.,  all $k$-chromatic graphs with $2k+2$ vertices are $k$-choosable. This conjecture was confirmed in \cite{zhuzhu}, where the following result was proved.
	
	\begin{theorem}
		\label{thm-zz}
		If $G$ is a complete $k$-partite non-$k$-choosable graph with $2k+2$ vertices, then  $k$ is even and  $G = K_{4, 2\star (k-1)}$ or $G= K_{3\star (k/2+1), 1 \star (k/2 -1)}$.
	\end{theorem}  
	
	Thus any $k$-chromatic non-$k$-choosable graph with $2k+2$ vertices is a spanning subgraph of  $G = K_{4, 2\star (k-1)}$ or $G= K_{3\star (k/2+1), 1 \star (k/2 -1)}$,  and  $k$ is an even integer.
	
	For $G = K_{4, 2\star (k-1)}$ or $G= K_{3\star (k/2+1), 1 \star (k/2 -1)}$,  a {\em bad list assignment } of $G$ is a $k$-list assignment $L$ of $G$ such that $G$ is not $L$-colourable. 
	Bad list assignments for $G=K_{4, 2 \star (k-1)}$ were characterized in \cite{EOOS2002}.  We shall give a simpler proof of this result. Then we characterize bad list assignments for  $G=K_{3 \star (k/2+1), 1 \star (k/2-1)}$:   If $k \ge 4$, then  a $k$-list assignment $L$ of $G=K_{3 \star (k/2+1), 1 \star (k/2-1)}$ is bad if and only if  
	\begin{itemize}
		\item $ | \bigcup_{v \in V(G)} L(v)| = \frac 32 k$,
		\item for each $3$-part $P$ of $G$, $\bigcap_{v \in P}L(v) =\emptyset$.
	\end{itemize}
	If $k=2$, then  a $k$-list assignment $L$ of $K_{3,3}$ is bad if and only if it is isomorphic to one of the list assignments in Figure \ref{1}.
	
	As a consequence,   we characterize all non-$k$-choosable $k$-chromatic graphs with $2k+2$ vertices.

	\begin{corollary}
		\label{col-1}
		For $k \ge 3$, every $k$-chromatic non-$k$-choosable graph $G$ with $2k+2$ vertices satisfies   $ \overline{K_4}  \vee     (2k_{k-1}) \subseteq G \subseteq K_{4, 2\star (k-1)}$, where $k$ is even,  $\overline{K_4}  \vee    (2k_{k-1}) $ is the join of   $\overline{K_4}$, an independent set of size $4$,  and $2K_{k-1}$, which  is the disjoint union of two copies of $K_{k-1}$.
	\end{corollary}

	\section{Some notation and preliminaries}
	
	Assume that $G = K_{4, 2\star (k-1)}$ or $G= K_{3\star (k/2+1), 1 \star (k/2 -1)}$, and $L$ is a bad  $k$-list assignment of $G$.
	
	For a subset $X$ of $V(G)$, let 
	$$L(X) = \bigcup_{v \in X} L(v).$$
	Let $C= L(V(G))$. It was proved in   \cite{Kierstead} that if $L$ is a bad list assignment of $G$ with $|C|$ minimum, then $|C| <   |V(G)| $. We observe that by the same argument, without assuming the minimality of $|C|$, the conclusion still holds for bad list assignments of $G$. Indeed,  if $|C| \ge |V(G)|$, then we let $X$ be a maximum subset of $V(G)$ with $|X| > |L(X)|$ (it is possible that $X = \emptyset$).  As $|V(G)| \le |C|$, we know that $X \ne V(G)$. Thus $|X| \le 2k+1$ and by Noel-Wu-Reed Theorem, we have
	$G[X]$ is $L$-colourable. The maximality of $X$ implies that for any $Y \subseteq V(G)-X$, 
	$|L(Y)-L(X)| \ge |Y|$. By Hall
	's Theorem, there is an injective colouring $f$ of $G-X$ such that for each vertex $v$, $f(v) \in L(v)-L(X)$. 
	Hence $G$ is $L$-colourable.
	
	Each  partite set of   $G$ is called a {\em part} of $G$, and a   part of size $i$ (respectively, at least $i$ or at most $i$) is called a {\em $i$-part} (respectively,{\em  $i^+$-part}, or {\em $i^-$-part}).
	
	For each $3^+$-part $P$ of $G$, let 
	$$t_P = \max\{|L(u ) \cap L( v)|: u \ne v, u,v \in P\}.$$
	
	For $c \in C$ and $C' \subseteq C$, let $$L^{-1}(c)=\{v: c \in L(v)\}, \ L^{-1}(C') = \bigcup_{c \in C'}L^{-1}(c).$$

	
	For a part $P$ of $G$ and integer $i$, let 
	$$C_{P,i} = \{c \in C: |L^{-1}(c) \cap P| = i \}, C_{P,i^+} = \{c \in C: |L^{-1}(c) \cap P| \ge i \}.$$

	\begin{definition}
		\label{def-gs}
		Assume $\mathcal{S}$  is a partition of $V(G)$, in which each part $S \in \mathcal{S}$ is an independent set.
		We denote by $\GS$ the graph obtained from $G$ by identifying each part $S \in \mathcal{S}$ into a single vertex $v_S$. Let $\LS$ be the list assignment of $G/\mathcal{S}$ defined as $\LS(v_S)=\bigcap_{v \in S}L(v)$.
	\end{definition} 
	
	If $S=\{v\}\in \mathcal{S}$ is a singleton part of $\mathcal{S}$, then we  may also denote   $v_S$ by $v$.
	In this case, $\LS(v)=L(v)$. In the partitions $\mS$ constructed in this paper, most parts of $\mS$ are singleton parts. To define $\mS$, it suffices to list  its non-singleton parts.

	\begin{definition}
		\label{def-bs}
		Let $\BS$ be the bipartite graph with partite sets $V(\GS)$ and $C$, in which $\{v_S,c\}$ is an edge if and only if $c \in \LS(v_S)$. 
	\end{definition}			
	
	It is obvious that a matching $M$ in $\BS$ covering $V(\GS)$ induces a proper $\LS$-colouring of $\GS$, which in turn induces a proper $L$-colouring of $G$. As $G$ is not $L$-colourable,   no such matching $M$ exists. By Hall's Theorem,  there is a subset $\XS$ of $V(\GS)$ such that 
	$|\XS| > |N_{\BS}(\XS)|$. 
	
	In the remainder of the  paper, for a given  partition  $\mathcal{S}$ of $V(G)$, let  $\XS$ be a maximum subset of $V(\GS)$ for which  $|\XS| > |N_{\BS}(\XS)|$. 
	Let $$\YS=	N_{\BS}(\XS)=\bigcup_{v_S \in \XS}\LS(v_S).$$ 
	
	\begin{observation}
		\label{obs-0} The following facts will be used often in this paper. 
		\begin{enumerate}
			\item[(1)] No two vertices in the same part of $G$ have the same list, and no colour is contained only  in the lists of vertices in  a same part. Thus for any part $P$ of $G$,    $\bigcup_{v \in V(G)-P}L(v)=C $.
			\item[(2)] Each $2^+$-part has no common colour.
		\end{enumerate}
	\end{observation}
	
	\begin{proof}
		(1) is trivial.
		
		(2)   Assume $P$ is a $2^+$-part of $G$ with $\bigcap_{v\in P} L(v)\neq \emptyset$, say $c\in  \bigcap_{v\in P} L(v)$. We colour all vertices of $P$ by colour $c$. Let $G'=G-P$ and $L'(v)=L(v)-\{c\}$ for any vertex $v$ of $G'$. As $|V(G')| \le 2\chi(G')+2$ and $\chi(G')=k-1$ is odd, it follows from Theorem \ref{thm-zz} that $G'$ is $L'$-colourable, and hence  $G$ is $L$-colourable, a contradiction.
	\end{proof}

	The following lemma was proved in \cite{zhuzhu} and also will be used in this paper.
	\begin{lemma}
		\label{ind3}
		Let $G$ be a complete multipartite graph with parts of size at most $3$. Let  $\mathcal{A},$  $\mathcal{D}$, $\mathcal{B}$, $\mathcal{C}$ be a partition of the set of parts of $G$ into classes  such that $\mathcal{A}$ and $\mathcal{D}$ contains only parts of size $1$, while $\mathcal{B}$ contains all parts of size $2$  and $\mathcal{C}$ contains all parts of size $3$. Let $k_1, d, k_2, k_3$ denote the cardinalities of classes $\mathcal{A}$, $\mathcal{D}$, $\mathcal{B}$, $\mathcal{C}$ respectively. Suppose that classes $\mathcal{A}$ and $\mathcal{D}$ are ordered, i.e.\ $\mathcal{A}= (A_1, \ldots, A_{k_1})$
		and $\mathcal{D}=(D_1, \ldots, D_d)$. If $f:V(G)\to \NN$ is a function for which the following conditions hold
		\begin{alignat}{2}
		f(v)&\geq k_2+k_3+i, &&\quad\text{for all $1\leq i \leq k_1$ and $v\in A_i$}\tag{a-1}\label{invK:a-1}\\
		f(v)&\geq 2 k_3+k_2+k_1+i, &&\quad\text{for all $1\leq i \leq d$ and $v\in D_i$}\tag{d-1}\label{invK:s-1}\\
		f(v)&\geq k_2+k_3, &&\quad\text{for all $v\in B \in\mathcal{B}$}\tag{b-1}\label{invK:b-1}\\
		f(u)+f(v)&\geq 3k_3+2k_2+k_1+d, &&\quad\text{for all $u,v\in B\in\mathcal{B}$}\tag{b-2}\label{invK:b-2}\\
		f(v)&\geq k_2+k_3, &&\quad\text{for all $v\in C\in\mathcal{C}$}\tag{c-1}\label{invK:c-1}\\
		f(u)+f(v)&\geq 2 k_3 + 2k_2+ k_1, &&\quad\text{for all $u,v\in C\in\mathcal{C}$}\tag{c-2}\label{invK:c-2}\\
		\sum_{v\in C}f(v)&\geq 4 k_3+ 3  k_2+2k_1+d-1, &&\quad\text{for all $C\in\mathcal{C}$} \tag{c-3}\label{invK:c-3}
		\end{alignat}
		then $G$ is $f$-choosable.	
	\end{lemma}
	
	\section{Bad list assignments for $K_{4, 2\star(k-1)}$}
	The following theorem is the characterization of bad list assignments of $K_{4, 2\star(k-1)}$  which has been proved in \cite{EOOS2002}. In this section, we give an alternate (and shorter) proof of this theorem. 
	
	\begin{theorem}
		\label{unique4}
		Assume $L$ is a bad  $k$-list assignment of $G=K_{4, 2\star(k-1)}$. Assume  $P_1=\{u_1, v_1, x_1, y_1\}$  and  $P_i=\{u_i, v_i\}$ for $2\le i\le k$ and $C=\bigcup_{v \in V(G)}L(v)$. Then  $C$ can be partitioned into $A$ and $B$ with $|A|=|B|=k$, where $A$ can be further partitioned into $A_1,A_2,A_3,A_4$ such that $|A_1|=|A_2|$, $|A_3|=|A_4|$ (where $A_3,A_4$ maybe empty),  and  $B$   can be further partitioned into $B_1,B_2$ with $|B_1|=|B_2|$, and the following hold:
		\begin{itemize}
			\item $L(u_1)=A_1\cup A_3\cup  B_1$, $L(v_1)=A_1\cup A_4\cup  B_2$, $L(x_1)=A_2\cup  A_4\cup B_1$, $L(y_1)=A_2\cup A_3\cup  B_2$.
			\item $L(u_i) =A$, $L(v_i)= B$ for $2\le i  \le k$.
		\end{itemize}	
	\end{theorem}
	
	\begin{proof}    
		If there is a colour $c\in C$ such that $|L^{-1}(c)\cap P_1|\ge 3$, then we colour vertices $L^{-1}(c)\cap P_1$ by colour $c$. 
		Let $G'=G-(L^{-1}(c) \cap P_1)$ and    $L'(v)=L(v)-\{c\}$ for $v \in V(G')$. It is easy to verify that $G'$ and $L'$ satisfy the condition of Lemma \ref{ind3} (here we need to use the fact that any 2-part $P$ has a vertex $v$ with $c \notin L(v)$, see (2) of  Observation \ref{obs-0}).  Therefore $G'$ is $L'$-colourable,  and hence $G$ is $L$-colourable, a contradiction. 
		
		Thus $|L^{-1}(c)\cap P_1|\le 2$ for any $c\in C$.
		By (2) of Observation \ref{obs-0}, we have $ |C| \ge 2k$. Recall that $|C| \le |V(G)|-1 = 2k+1$. Depending on the size of $C$, we consider two cases.
		
		\medskip
		\noindent
		{\bf Case 1} $|C|=2k+1$.
		\medskip 
		

		Since $|C|=2k+1$ and $|L(u_1)|+|L(v_1)|+|L(x_1)|+|L(y_1)|\ge 4k$, we may assume 
		$ L(u_1 ) \cap L( v_1) \ne \emptyset$. Let $\mS$ be the partition of $V(G)$ with one non-singleton part $S=\{u_1,v_1\}$, and consider the graph $\GS$ and list assignment $\LS$. Let $\XS$ and $\YS$ be as defined in Section 2.  We denote by $P'_i$ be the parts of $\GS$.
		
		It is obvious that $\XS-\{v_S\} \ne \emptyset$, and hence   $|\YS|\ge |L(v)| \ge k$ for $v \in \XS-\{v_S\}$.   Hence   $|\XS|\ge k+1$. If there is an index $i \ge 2$ such that $P'_i\subseteq \XS$, then $|\YS|\ge |L'(u_i)|+|L'(v_i)|\ge 2k$ and hence $|\XS|=|V(\GS)|=2k+1$. But in this case, by Observation \ref{obs-0}, $|\YS|=|C|=2k+1=|\XS|$, a contradiction.
		
		So $|P'_i\cap X|\le 1$ for $i \ge 2$, and hence $|\XS| \le k+2$. Since $|\XS|\ge k+1$, $|\XS\cap P'_1|\ge 2$. This implies that $|\YS|\ge k+1$ and hence $|\XS|= k+2$, i.e., $P'_1\subseteq X$. So $|\YS|\ge |L'(v_S)|+|L'(x_1)\cup L'(y_1)|\ge 1+k+1=k+2 =|\XS|$,   a contradiction.

		\medskip
		\noindent
		{\bf Case 2} $|C|=2k$.
		\medskip

		As $|C_{P_1,1}|+2|C_{P_1,2}| = \sum_{v \in P_1}|L(v)|=4k$ and
		$|C_{P_1,1}|+|C_{P_1,2}|\le |C|$,  we conclude that  $|C_{P_1,2}|=2k$, i.e., $C_{P_1,2}=C$.  
		
		\begin{claim}
			\label{clm-2subset}
			For any 2-subset $U$ of $P_1$, $|\bigcap_{v \in U} L(v)| = |\bigcap_{v \in P_1-U}L(v)|$.
		\end{claim}
		\begin{proof}
			Assume $U$ is a 2-subset   of $P_1$. Then 
			$(\bigcap_{v \in U} L(v)) \cap (\bigcup_{v \in P_1-U} L(v)) = \emptyset$. So
			$$2k=|C| \ge |\bigcap_{v \in U} L(v)|+ |\bigcup_{v \in P_1-U} L(v)| =  |\bigcap_{v \in U} L(v)| + 2k -|\bigcap_{v \in P_1-U} L(v)|.$$
			Hence 
			$|\bigcap_{v \in U} L(v)| \le |\bigcap_{v \in P_1-U}L(v)|$. By symmetry, 	$|\bigcap_{v \in P_1-U} L(v)| \le |\bigcap_{v \in U}L(v)|$. So
			$|\bigcap_{v \in U} L(v)| = |\bigcap_{v \in P_1-U}L(v)|$.
		\end{proof}	
		
		\begin{claim}
			\label{clm-u1v1}
			Assume $U$ is a 2-subset  of $P_1$  with  $\bigcap_{v \in U} L(v) \ne \emptyset$.  Then there is a $k$-subset $Y_U$ of $C$ and a $(k-1)$-subset $S_U$ of $V(G)$ such that for $i=2,3,\ldots, k$,  $|S_U \cap P_i|=1$  and  
			$$L(S_U \cap P_i)=L(S_U) \cup (\bigcap_{v \in U} L(v)) \cup (\bigcap_{v \in P_1-U} L(v)) = Y_U.$$
		\end{claim}
		\begin{proof}
			Assume $U=\{u_1, v_1\}$ and $L(u_1 ) \cap L( v_1) \ne \emptyset$. By Claim \ref{clm-2subset}, we know that 
			$L(x_1 ) \cap L( y_1) \ne \emptyset$. Let $\mS$ be the partition of $V(G)$ whose non-singleton parts are $S=\{u_1,v_1\}, T=\{x_1, y_1\}$. 
			
			
			It is easy to see that $\XS - \{v_S,v_T\} \ne \emptyset$. Hence $|\YS| \ge |L(v)| \ge k$ for some $v \in \XS-\{v_S,v_T\}$. This implies that $|\XS| \ge k+1$.  If $|\XS| > k+1$, then    $P_i\subseteq X$ for some $2\le i\le k$, and hence $|\YS|\ge 2k =|V(\GS)| \ge |\XS|$, a contradiction. So $|\XS|=k+1$ and $k=|\YS|$. 
			This implies that $P'_1 \subseteq \XS$ and 
			$|\XS \cap P_i| =1$ for $2 \le i \le k$. 
			Let $S_U=\XS-\{v_S, v_T\}$ and $Y_U=\YS$. We have
			$ L(S_U) \cup (\bigcap_{v \in U} L(v)) \cup (\bigcap_{v \in P_1-U} L(v)) = Y_U.$ For $i=2,\ldots, k$ and $v \in S_U \cap P_i$, since $|Y_U|=k \le  |L(v)|$, we have $Y_U=L(v)$. 
		\end{proof}
		
		
		\begin{corollary}
			For distinct 2-subsets $U$ and $U'$ of $P_1$, either $S_{U} = S_{U'}$ and $Y_{U} = Y_{U'}$, or 
			$S_{U} \cap S_{U'} = \emptyset$ and $Y_U \cap Y_{U'} = \emptyset$. 
		\end{corollary}

		By symmetry, we may assume  $L(u_1 ) \cap L( v_1) \ne \emptyset$, $S_{\{u_1,v_1\}}=\{u_2,u_3, \ldots, u_k\}$. 
		
		Since $|L(u_1 ) \cap L( v_1)|= |L(x_1 ) \cap L( y_1)|$
		and $|L(u_1 ) \cap L( v_1)| + |L(x_1 ) \cap L( y_1)| \le |Y_{\{u_1,v_1\}}|=k$, we have 
		$|L(u_1 ) \cap L( v_1)|= |L(x_1 ) \cap L( y_1)| \le k/2$. 
		
		As $C_{P_1,2} =C$,  we may assume that $L(u_1 ) \cap L( x_1) \ne \emptyset$. Similarly, we have 
		$|L(u_1 ) \cap L( x_1)|= |L(v_1 ) \cap L( y_1)| \le k/2$.

		\medskip
		\noindent
		{\bf Case 2(a)}   $L(u_1 ) \cap L( y_1) = \emptyset$.
		\medskip
		
		Then $$C_{P_1,2}=C = (L(u_1 ) \cap L( v_1)) \cup (L(x_1 ) \cap L( y_1)) \cup 
		(L(u_1 ) \cap L( x_1)) \cup (L(v_1 ) \cap L( y_1)).$$
		Hence  $|L(u_1 ) \cap L( v_1)|= |L(x_1 ) \cap L( y_1)| =|L(u_1 ) \cap L( x_1)|= |L(v_1 ) \cap L( y_1)| = k/2$.
		Then $$(L(u_1 ) \cap L( v_1)) \cup (L(x_1 ) \cap L( y_1))=Y_{\{u_1,v_1\}},$$ $$(L(u_1 ) \cap L( x_1)) \cup (L(v_1 ) \cap L( y_1)) = Y_{\{u_1,x_1\}}.$$
		As $L(u_1 ) \cap L( v_1) \cap L(x_1) = \emptyset$, we conclude that $Y_{\{u_1, v_1\}} \cap Y_{\{u_1,x_1\}} = \emptyset$ and $S_{\{u_1, v_1\}} \cap S_{\{u_1,x_1\}} = \emptyset$. 
		
		Thus we have $S_{\{u_1, x_1\}} = \{v_2,v_3, \ldots, v_k\}$. Theorem \ref{unique4} holds with $A_1=L(u_1 ) \cap L( v_1), A_2 = L(x_1 ) \cap L( y_1), A_3=A_4 = \emptyset$, $B_1=L(u_1 ) \cap L( x_1), B_2= L(v_1 ) \cap L( y_1)$.

		\medskip
		\noindent
		{\bf Case 2(b)}   $L(u_1 ) \cap L( y_1) \ne \emptyset$.
		\medskip
		
		Since 
		\begin{equation*}
		\begin{aligned}
		C_{P_1,2}=C &= (L(u_1 ) \cap L( v_1)) \cup (L(x_1 ) \cap L( y_1)) \cup 
		(L(u_1 ) \cap L( x_1)) \cup (L(v_1 ) \cap L( y_1)) \\
		& \cup 
		(L(u_1 ) \cap L( y_1)) \cup (L(v_1 ) \cap L( x_1)),
		\end{aligned}
		\end{equation*} 
		we may assume that $S_{\{u_1,v_1\}} = S_{\{u_1 y_1\}}$,
		$Y_{\{u_1,v_1\}} = Y_{\{u_1 y_1\}}$,
		$S_{\{u_1,v_1\}} \cap S_{\{u_1 x_1\}} = \emptyset$
		and $Y_{\{u_1,v_1\}} \cap Y_{\{u_1 x_1\}} = \emptyset$.
		Then Theorem \ref{unique4} holds with $A_1=L(u_1 ) \cap L( v_1), A_2 = L(x_1 ) \cap L( y_1), A_3=L(u_1 ) \cap L( y_1),
		A_4 = L(v_1 ) \cap L( x_1) $, $B_1=L(u_1 ) \cap L( x_1), B_2= L(v_1 ) \cap L( y_1)$.
		
		This completes the proof of Theorem \ref{unique4}.
	\end{proof}

	\section{Bad list assignments for  $K_{3\star (k/2+1), 1\star(k/2-1)}$}
	
	Let $G= K_{3\star (k/2+1), 1\star(k/2-1)}$.
	Assume $L$ is a bad $k$-list assignment of $G$. By Observation \ref{obs-0}, for any 3-part $P$ of $G$, 
	$\bigcap_{v \in P}L(v) = \emptyset$. Thus 
	$2|C| \ge \sum_{v \in P}|L(v)| \ge 3k$. Hence 
	$|C| \ge 3k/2$. If $|C| = 3k/2$, then $G$ is not $L$-colourable, because if $f$ is a proper $L$-colouring of $G$, then $|f(P)| \ge 2$ for each 3-part $P$ and $|f(P)|=1$ for each 1-part $P$. As $f(P) \cap f(P') = \emptyset$ for distinct parts $P$ and $P'$, we have 
	$|C| \ge \sum |f(P)| \ge 2 \times (k/2+1) + (k/2-1) = 3k/2+1$, a contradiction. 
	
	If $k=2$, then there are some other bad $2$-list assignment of $K_{3,3}$.

	\begin{theorem}
		Any bad $2$-list assignment of $K_{3,3}$ is isomorphic to one of the list assignments in Figure \ref{1}.
	\end{theorem}
	\begin{proof}
		Assume $G=K_{3,3}$ and $L$ is a bad $2$-list assignment of $G$.
		Assume the  two parts of $G$ are $P_i=\{u_i,v_i,w_i\}$ for $i=1,2$. By (2) of Observation \ref{obs-0}, $C_{P_i,3} = \emptyset$ for $i=1,2$. 
		As $|C| \le 5$, we know that $C_{P_1,2}   \ne \emptyset$.
		Assume $1 \in L(u_1) \cap L(v_1)$. If 
		$1 \notin C_{P_2,2}$, then we colour 
		$u_1,v_1$ by colour $1$,  which easily extends to an $L$-colouring of $G$, a contradiction. 
		
		Thus we may assume that  $1 \in L(u_1) \cap L(v_1) \cap L(u_2) \cap L(v_2)$.
		If $L(w_1) \not\subseteq   (L(u_2) \cup L(v_2) )$, then we colour $\{u_1,v_1\}$ by colour $1$, and colour $w_1$ by a colour in $L(w_1) -  (L(u_2) \cup L(v_2) )$, which extends to a proper $L$-colouring of $G$. Thus we must have  $L(w_1) =  (L(u_2) \cup L(v_2) ) -\{1\}$, and by symmetry, $L(w_2)=  (L(u_1) \cup L(v_1) ) -\{1\}$. 
		Thus $L(u_1) = \{1,a\}, L(v_1)= \{1,b\}, L(w_1)=\{c,d\}, L(u_2)=\{1,c\}, L(v_2)=\{1,d\}, L(w_2)=\{a,b\}$, where $a \ne b$ and $c \ne d$. Depending on 
		$|\{a,b\} \cap \{c,d\}|=0,1$ or $2$, we have three bad 2-list assignments for $K_{3,3}$ as depicted in Figure \ref{1}.   
	\end{proof}

	\begin{figure}[htbp]
		\centering
		\includegraphics[width=2in]{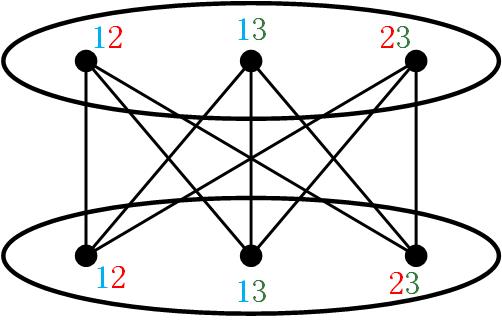}
		\includegraphics[width=2in]{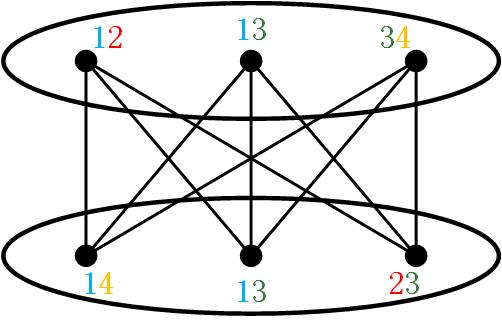}
		\includegraphics[width=2in]{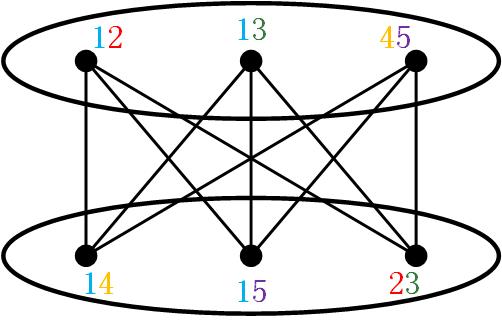}
		\caption{ all bad $2$-list assignments of $K_{3,3}$}
		\label{1}
	\end{figure}

	\begin{theorem}
		\label{unique3}
		Suppose $k\ge 4$ is an even integer and   $L$ is the $k$-list assignment of $G$ with  $|C|>3k/2$. Then $G$ is $L$-colourable. 
	\end{theorem}
	\begin{proof}
		For an ordering   $\pi=(P_1, P_2, \ldots, P_{k/2+1})$ of the 3-parts, let
		
		$$\ell(\pi) = \max\{i: t_{P_j} \ge j,   \forall j \le i\}.$$
		
		Let $\pi$ be an ordering of the 3-parts   such that 
		\begin{itemize}
			\item[(1)] $\ell(\pi)$ is maximum.
			
			\item[(2)] Subject to (1), $|  L(P_{k/2+1})|$ is maximum.
		\end{itemize} 
		
		
		For $1 \le i \le \ell(\pi)$, choose a 2-subset $S_i$ of $P_i$ with $\bigcap_{v \in S_i}L(v) = t_{P_i}$. 
		Let $\mS$ be the partition of $V(G)$ whose non-singleton parts are $S_1,S_2, \ldots, S_{\ell(\pi)}$. Let $\GS, \LS, \XS, \YS$ be as defined in Section 2. Similarly, let $P'_i$ be the parts of $\GS$.
		
		Let $  Z=\{v_{S_i}: 1 \le i \le \ell(\pi)\}.$
		Subject to the condition that $\bigcap_{v \in S_i}L(v) = t_{P_i}$, we choose $S_i$ 
		so that $|\LS(Z)|$ is maximum.

		\begin{claim}
			\label{clm-33}
			$\XS -Z \ne \emptyset$, $\ell(\pi)\le k/2$ and 	$|\XS \cap P_i| \ge 2$ for some $\ell(\pi)+1 \le i \le k/2+1$.
		\end{claim}
		\begin{proof}
			If $\XS -Z = \emptyset$, then let $i$ be the maximum index such that $v_{S_i}\in \XS$, then   $|\YS|\ge i\ge |\XS|$, a contradiction.
			
			So  $|\YS| \ge |L(v)| = k$ for $v \in \XS -Z$. Hence $|\XS| \ge k+1$. 
			Thus for some  part $P'_i$ of $\GS$, $|P'_i \cap \XS|\ge 2$.
			
			If 	$|\XS \cap P_i| \le 1$ for each $i \ge \ell(\pi)+1$, then 
			let $i$ be the maximum index such that $P'_i \subseteq \XS$. Then $|\XS| \le k+i$. But 
			$|\YS| \ge |\LS(P'_i)| \ge k+i$, a contradiction. Thus $\ell(\pi) \le k/2$ and 	$|\XS \cap P_i| \ge 2$ for some $\ell(\pi)+1 \le i \le k/2+1$.
		\end{proof}
		
		\begin{claim}
			\label{clm-ell}
			$|\YS| \ge 3k/2$, $|V(\GS)-\XS| \le 1$ and  $\ell(\pi)= t_{P_{k/2+1}}=k/2$.
		\end{claim}
		\begin{proof}
			By Claim \ref{clm-33}, there exists 	 $\ell(\pi)+1 \le i \le k/2+1$ such that 
			$|\XS \cap P_i| \ge 2$. Assume 
			$x,y  \in \XS \cap P_i$. 
			Then
			\begin{equation}
			\label{eq-9}
			|\YS| \ge |L(x )\cup L( y)| \ge 2k- t_{P_i} \ge 2k-\ell(\pi) \ge 3k/2.
			\end{equation}   
			Hence
			\begin{equation*}
			2k+2   -  \ell(\pi) =|V(\GS)| \ge |\XS| \ge 2k- \ell(\pi)+1.
			\end{equation*}  
			
			So   $|\XS| \ge 3k/2+1$. This implies that either     $\XS=\GS$ or $\XS=\GS-v^*$ for some vertex $v^*$. 
			
			If $\ell(\pi)\le k/2-1$, then $\XS$ contains a 3-part $P_i$, where $i \ge \ell(\pi)+1$. By the maximality of $\ell(\pi)$, we have $|L(x )\cap L( y)| \le \ell(\pi)$ for any $x,y \in P_i$. 
			Hence
			$$3\ell(\pi)\ge|C_{P_i,2}|\ge 3k-|L(P_i)|\ge 3k-|\YS| \ge k+\ell(\pi)-1,$$
			a contradiction. Hence $\ell(\pi) \ge k/2$, and by Claim \ref{clm-33}, we have $
			\ell(\pi)=k/2$.
			
			Now we show that 
			$  t_{P_{k/2+1}} =  k/2$. It follows from the definition of $
			\ell(\pi)$ that $t_{P_{k/2+1}} \le k/2$. Assume to the contrary that $t_{P_{k/2+1}} \le k/2 -1$. Then $|\YS| \ge |L(x) \cup L(y)| \ge 2k-t_{P_{k/2+1}} \ge 3k/2+1$. Hence $|\XS| \ge 3k/2+2$, which implies that $\XS = V(\GS)$. In particular, $P_{k/2+1} \subseteq \XS$ and hence $|\YS| \ge |L(P_{k/2+1})| \ge 3k-3t_{P_{k/2+1}} \ge 3k/2+3 > |\XS|$, a contradiction. 
		\end{proof}
		
		\begin{corollary}
			\label{cor-ell}
			For $i=1,2,\ldots, k/2+1$, $t_{P_i} = k/2$.
		\end{corollary}
		\begin{proof}
			Interchange the positions of $P_i$ and $P_{k/2+1}$, we obtain an ordering $
			\pi'$ of the parts of $G$ with $\ell(\pi') = k/2$. Apply Claims \ref{clm-33} and \ref{clm-ell} to $\pi'$, we conclude that $t_{P_i} = k/2$. (Note that the proofs of these two claims only used the fact that $\pi$ is an ordering with $\ell(\pi)$ maximum among all orderings of the parts of $G$.)
		\end{proof}

		In the following, we assume that 
		$\XS=\GS$ or $\XS=\GS-v^*$ for some vertex $v^*$. Note that for any $i \le k/2$, if $\pi'$ is the ordering of the parts of $G$ obtained from $\pi$ by  interchanging the position of $P_i$ and $P_{k/2+1}$, we have $\ell(\pi') = k/2 = \ell(\pi)$. Thus by the choice of $\pi$, we have
		\begin{equation}
		\label{eqn-lp}
		|L(P_i)| \le |L(P_{k/2+1})|, \forall i \le k/2.
		\end{equation}  
		\begin{claim}
			\label{samecolour}
			If $P_{k/2+1} \subseteq \XS$, then $|\YS| > 3k/2$
		\end{claim}
		\begin{proof}
			Assume to the contrary that $P_{k/2+1}\subseteq \XS$ and  $|\YS| = 3k/2$.
			
			For $i=1,2,\ldots, k/2+1$, $$2|L(P_i)|\ge |C_{P_{i},1}|+2|C_{P_{i},2}|=\sum_{v\in P_{i}}|L(v)|\ge 3k$$ and hence $|L(P_i)|\ge 3k/2$. Combined with $|L_{P_{k/2+1}}|\le |\YS|=3k/2$, we have $|L(P_{k/2+1})|=3k/2$. 
			
			By (\ref{eqn-lp}), for any  $j \le k/2$,   $|L(P_j)| = 3k/2$. This  implies that 
			$L(P_j)=\LS(P'_j)$. 
			
			Assume $v^* \in P_i$ (if $v^*$ exists). Then $C = \bigcup_{v \in V(G)-P_i}L(v) \subseteq \YS$. This is a contradiction, as $|C| > 3k/2$.   
		\end{proof}

		In the remaining part of the  proof, we consider two cases.

		\medskip\noindent
		{\bf Case 1}:   $\XS=\GS-v^*$.

		In this case, $|\XS|=3k/2+1$ and $|\YS|=3k/2$. 
		By Claim \ref{samecolour}, we conclude that $P_{k/2+1} \not\subseteq \XS$ and hence  $v^* \in P_{k/2+1}$. By the maximality of $\XS$, we know that 
		\begin{equation}
		\label{eq-v^*}
		|L(v^*)-\YS|\ge 2.
		\end{equation}
		
		Assume $P_{k/2+1} = \{ u_{k/2+1}, v_{k/2+1}, w_{k/2+1}\}$ and $v^*=v_{k/2+1}$. 
		Since $|L(u_{k/2+1} )\cup L( w_{k/2+1})|\le |\YS|= 3k/2$, we have $|L(u_{k/2+1} )\cap L( w_{k/2+1})| \ge k/2$. As $\ell(\pi)=k/2$, we have  $|L(u_{k/2+1} )\cap L( w_{k/2+1})|= k/2$.
		
		As $2|P_{1,2}|+|P_{1,1}| \ge3k$ and $|P_{1,2}|+|P_{1,1}| \le |C| \le 2k+1$, we have $|P_{1,2}| \ge k-1 > k/2$. Hence there is a 2-subset $S'_1$ of $P_1$ such that $$\bigcap_{v \in S'_1} L(v) \not\subseteq L(u_{k/2+1} )\cap L( w_{k/2+1}).$$
		
		Let $S_{k/2+1} = \{u_{k/2+1} , w_{k/2+1} \}$ and let
		$\mS' = (\mS - \{S_1\}) \cup \{S'_1, S_{k/2+1}\}$.
		
		Let $P''_i$ be the part of $G/\mathcal{S}'$ and let $Z'=(Z- \{v_{S_1}\}) \cup \{v_{S'_1}, v_{S_{k/2+1}}\}$.

		We have $$X_{\mathcal{S'}}-Z'\neq \emptyset,$$
		for otherwise, if $|X_{\mathcal{S'}}\cap Z'|\le k/2$, then  $|X_{\mathcal{S'}}| \le k/2\le |Y_{\mathcal{S'}}| $  (as it is obvious that $|\XS| \ge 2$), a   contradiction, and  if   $X_{\mathcal{S'}}=Z'$, then 
		$|X_{\mathcal{S'}}|= k/2+1 \le |L_{\mathcal{S'}}(v_{S'_1}) \cup L_{\mathcal{S'}}(v_{S_{k/2+1}})| \le |Y_{\mathcal{S'}}|$,  
		a contradiction.
		
		Hence, there is a vertex $v\in X_{\mathcal{S'}}$ such that $|L_{\mathcal{S'}}(v)|\ge k$ and this implies that $|X_{\mathcal{S'}}|\ge k+1$. So, we know that there exists an index $1\le j\le k/2+1$ such that $P''_j\subseteq X_{\mathcal{S'}}$. 
		
		This implies that $|Y_{\mathcal{S'}}|\ge k+1$, and hence  $|X_{\mathcal{S'}}| \ge k+2$. This in turn implies that  there exists an index $2 \le j\le k/2+1$ such that $P''_j\subseteq X_{\mathcal{S'}}$. 
		This implies that  $|Y_{\mathcal{S'}}|\ge
		3k/2$. Hence $|X_{\mathcal{S'}}|= 3k/2+1$, i.e., $X_{\mathcal{S'}}=V(	G/\mathcal{S}')$ and $|Y_{\mathcal{S'}}|= 3k/2$.
		But we know that $|Y_{\mathcal{S'}}|\ge |\YS \cup L(v^*)|\ge 3k/2+2$ (by (\ref{eq-v^*})), a contradiction.

		\medskip\noindent
		{\bf Case 2}: $\XS=V(\GS)$

		In this case, $|\XS|=3k/2+2$. In particular, 
		\begin{equation}
		\label{eq-p}
		P_{k/2+1} \subseteq \XS, L(P_{k/2+1}) \subseteq \YS.
		\end{equation} 
		By  Claim \ref{samecolour},   $|\YS|= 3k/2+1$.

		\begin{claim}
			\label{clm-pi}
			$|\LS(Z)| \ge k/2+1$ and  $ L(P_i) \subseteq \YS$ for $1\le i\le k/2+1$. 
		\end{claim}
		\begin{proof}
			By Corollary \ref{cor-ell},  $t_{P_i}  =k/2$ for all 
			$1 \le i \le k/2+1$.  As  $|L(P_{k/2+1})| \le |\YS| =3k/2+1$, we have 
			$3k/2+1 \ge |L(P_{k/2+1})|  \ge |L(P_i)|$. This implies  that $P_i$ has at least two 2-subsets $S_i$ with $|\bigcap_{v\in S_i}L(v)|=k/2$. As $\ell(\pi) \ge 2$, we can make a choice of $S_2$ so that $\bigcap_{v\in S_1}L(v) \ne \bigcap_{v\in S_2}L(v)$. Hence $$|\LS(Z)| \ge |\bigcap_{v\in S_1}L(v) \cup \bigcap_{v\in S_2}L(v)| \ge k/2+1.$$

			If $ |L(P_{k/2+1})|=3k/2$, then $|L(P_i)|=|\LS(P'_i)|=3k/2$ for all $1 \le i \le k/2$,  and hence $L(P_i) = \LS(P'_i) \subseteq \YS$, and the claim is true.  
			
			Assume $|L(P_{k/2+1})|=3k/2+1$. Then $\YS=L(P_{k/2+1})$.
			
			Assume to the contrary that 
			$L(P_i) \not\subseteq \YS$ for some $1 \le i \le k/2$.
			Since $t_{P_i} = k/2$ for $1 \le i \le k/2$, by a re-ordering of the parts of $G$ if needed, we may assume that   $P_1=\{u_1,v_1,w_1\}$ and $L(u_1) - \YS \ne\emptyset$ (the resulting ordering $\pi'$ will still have $\ell(\pi')=k/2$). Let $S'_1=\{v_1,w_1\}$ and 
			$\mS' = (\mS -\{S_1\}) \cup \{S'_1\}$.   The argument remains   valid for $X_{\mathcal{S}'}$.  
			Hence $|X_{\mathcal{S}'}|=|V(G/\mathcal{S}')|= 3k/2+2$.	
			However, $\YS \cup (L(u_1) - \YS) \subseteq Y_{\mathcal{S}'}   $. Hence $|Y_{\mathcal{S}'}| \ge |\YS|+1 = 3k/2+2$, a contradiction. 
		\end{proof}

		By (2) of Observation \ref{obs-0}, $C= \bigcup_{i  \in [k] - \{ k/2+1\}} L(P_i)$.
		By Claim \ref{clm-pi}, $  \bigcup_{i \in [k] - \{k/2+1\}} L(P_i) \subseteq  \YS$. 
		
		Let $S_{k/2+1}$ be a 2-subset of $P_{k/2+1}$  with $\cap_{v \in S_{k/2+1}} L(v) = k/2$, and let  $\mS'=\mS \cup \{S_{k/2+1}\}$.
		By using the fact that $|L_{\mathcal{S'}}(Z)| =|\LS(Z)| \ge k/2+1$, it is easy to show that   $X_{\mathcal{S'}}= V(\GS')$, and hence  $|Y_{\mathcal{S'}}| \le |\YS| = 3k/2+1 =|X_{\mathcal{S'}}|$, a contradiction. 

	This completes the proof of Theorem \ref{unique3}.	

\end{proof}

\section{Characterization of non-$k$-choosable $k$-chromatic graphs with $2k+2$-vertices}

This section proves Corollary \ref{col-1}. By Theorem \ref{thm-zz}, it suffices to consider   proper subgraphs of $K_{3\star(k/2+1),1\star(k/2-1)}$ and $K_{4,2\star(k-1)}$, where $k$ is an even integer.

Let $G^*$ be the subgraph of $K_{4, 2\star (k-1)}$ obtained by deleting all the edges  $\{u_iv_j: 2 \le i,j \le k\}$.   In other words, $G^*=\overline{K_4}  \vee    (2k_{k-1}) $ is the join of   $\overline{K_4}$, an independent set of size $4$,  and $2K_{k-1}$, which  is the disjoint union of two copies of $K_{k-1}$. It is easy to verify that $G^*$ is not $k$-choosable. Indeed, let
$L$ be the $k$-list assignment of $G^*$ as described in Theorem \ref{unique4}. Assume $f$ is a proper $L$-colouring of $G^*$.  Then $|f(\{u_2,u_3,\ldots, u_k\})|=|f(\{v_2,v_3,\ldots, v_k\})|= k-1$. Hence  there is only one   colour from $A$, and only one colour from $B$,  that are left for vertices in $P_1$.  No matter what are the colours from $A$ and $B$ are left, the list of   one of the vertices of $P_1$ contains none of these two colours, and hence cannot be properly coloured by a colour from its list. 

Thus any subgraph of $K_{4, 2\star (k-1)}$ containing a copy of $G^*$ is not $k$-choosable. The following lemma shows that any other subgraph of $K_{4, 2\star (k-1)}$ is $k$-choosable. 

\begin{theorem}
	\label{thm-4proper}
	Assume $G$ is a subgraph of $K_{4,2\star(k-1)}$ which does not contain a copy of $G^*$. Then $G$ is $k$-choosable. 
\end{theorem}
\begin{proof}
	Assume $L$ is a $k$-list assignment of $G$.
	We may assume $L$ is the $k$-list assignment as described in Theorem \ref{unique4}, for otherwise, $K_{4,2\star(k-1)}$ is $L$-colourable, and hence
	$G$ is $L$-colourable. If there are $2 \le i < j \le k$ such that $u_iu_j$ is not an edge of $G$, then since $L(u_i) = L(u_j)$, we can identify $u_i$ and $u_j$ into a single vertex. It follows from Noel-Reed-Wu Theorem that $G$ is $L$-colourable. Thus we may assume that $\{u_2,u_3, \ldots, u_k\}$ induces a clique. Similarly, $\{v_2,v_3, \ldots, v_k\}$ induces a clique.
	
	As $G$ does not contain $G^*$ as a subgraph, we know that $uv \notin E(G)$ for some $u \in P_1$ and $v \in P_i$ for some $2 \le i \le k$. By symmetry, assume $u_1u_2 \notin E(G)$. We colour $u_1,u_2$ by a colour $c \in L(u_1 )\cap L( u_2) \subseteq A$ (which exists by the description of $L$ in Theorem \ref{unique4}), and colour $v_1$ and $y_1$ by a colour   $c' \in B_2$.

	Let $G'=G-\{u_1,u_2, v_1, y_1\}$ and let $L'(v)=L(v)-\{c,c'\}$ for $v \in V(G')$.  
	It easy to verify that $G'$ and $L'$ satisfies the conditions of Lemma \ref{ind3} and hence $G'$ is $L'$-colourable, and hence $G$ is $L$-colourable.
\end{proof}

For $k=2$, it is easy to verify (and also follows from the characterization of 2-choosable graphs in \cite{ERT1980}) that, up to isomorphism, $K_{3,3}$ has two proper subgraphs that are not 2-choosable.
In the following, we show that for $k \ge 4$, all proper subgraph of $K_{3\star (k/2+1), 1\star(k/2-1)}$ is $k$-choosable.

\begin{theorem}
	\label{thm-3proper}
	If $k \ge 4$ and $G$ is a proper  subgraph of $K_{3\star (k/2+1), 1\star(k/2-1)}$, then $G$ is $k$-choosable.
\end{theorem}
\begin{proof}
	Let $P_i=\{u_i,v_i,w_i\}$ for $1\le i\le k/2+1$ and $P_i=\{u_i\}$ for $k/2+2\le i\le k$.
	By Noel-Reed-Wu Theorem, we may assume $G=K_{3\star (k/2+1), 1\star(k/2-1)}-\{e\}$ for some edge   $e=xy$, say $x\in P_i$ and $y\in P_j$, $i <j$.  If $j \ge k/2+2$, then $G$ is a subgraph of $K_{3\star k/2, 2 \star 2, 1\star (k/2-2)}$ which is $k$-choosable by  Theorem \ref{thm-zz}.   Hence, $G$ is $L$-colourable, a contradiction.

	Thus we may assume $x=u_1$ and $y=u_2$.
	Since $|C|=3k/2$, we have $|L(u_1)\cap L(u_2)|\ge k/2$ and $|L(v_i)\cap L(w_i)|\ge k/2$ for $i=1,2$. Let $L(u_1)\cap L(u_2)=A_1$, $L(v_1)\cap L(w_1)=A_2$ and $L(v_2)\cap L(w_2)=A_3$. 
	
	Note that $L(u_1)\subseteq C-A_2$ and $L(u_2)\subseteq C-A_3$. Hence $|A_2|=|A_3|=k/2$, $L(u_1)=C-A_2$ and $L(u_2)= C-A_3$.	
	%
	This implies that
	$$|A_1\cup A_2\cup A_3|=|(C-A_2)\cap (C-A_3)|+|A_2\cup A_3|=|C-(A_2\cup A_3)|+|A_2\cup A_3|=|C|>k.$$
	So there exists  $c_i\in A_i$ such that $|L(u_k)- \{c_1,c_2,c_3\}|\ge k-2$.

	Now, we colour $u_1$, $u_2$ by colour $c_1$, colour $v_1$, $w_1$ by colour $c_2$ and colour $v_2$, $w_2$ by colour $c_3$. Let $L'(v)=L(v)-\{c_1,c_2,c_3\}$ for any vertex $v$ of $G'=G-P_1-P_2$. 
	
	It easy to verify that $G'$ and $L'$ satisfies the condition of Lemma \ref{ind3} ({for the verification of condition (c-2) , if $k\ge 6$, then it is trivial and if $k=4$, then we need to use the fact that $|L(u)\cap L(v)|\le 2$ for any $u,v$ of 3-part $P$.})  and hence $G'$ is $L'$-colourable, a contradiction.
\end{proof}

Corollary \ref{col-1} follows from Theorem \ref{thm-4proper} and Theorem \ref{thm-3proper}.

\end{document}